\documentclass[12pt]{amsart}
\usepackage{amssymb}
\usepackage{amsthm}
\usepackage{enumitem}
\newcommand{\R}{\mathbb R}
\newcommand{\E}{\mathbb E}

\newtheorem{definition}{Definition}
\newtheorem{theorem}{Theorem}
\newtheorem{lemma}{Lemma}

\theoremstyle{remark}
\newtheorem{remark}{Remark}
\begin{document}
\author{Tomasz Szostok}
\title{A generalization of a theorem of Brass and Schmeisser}
\address{Institute of Mathematics, University of Silesia in Katowice, 40-007 Bankowa 14, Katowice, Poland }
\email{tomasz.szostok@us.edu.pl}
\maketitle
\begin{abstract}
Let $n$ be an odd positive integer.
It was proved by Brass and Schmeisser that 
for every quadrature $$\mathcal{Q}=\alpha_1f(x_1)+\dots+\alpha_mf(x_m),$$ (with positive weights) of order at least $n+1$ and for every $n-$convex function $f,$ the value of $Q$ on $f$ lies between the values of Gauss and Lobatto quadratures of order $n+1$ calculated for the same function $f$. We generalize this result in two directions, replacing 
$Q$ by an integral with respect to a given measure and allowing the number $n$ to any positive integer (for even $n$ Radau quadratures replace Gauss and Lobatto ones).

\end{abstract}
\maketitle
\section{Introduction}
Theorem 6 from the paper of Brass and Schmeisser \cite{BS}
states that for every odd number $n,$ the quadratures of Gauss and Lobatto of order $n+1$ are extremal  with respect to $n-$convex ordering
i.e. for every quadrature$$\mathcal{Q}=\alpha_1f(x_1)+\dots+\alpha_nf(x_m),$$
of the same order
and for every $n-$convex function $f,$ the value of  $Q$
lies between the values of the quadratures of Gauss and Lobatto (calculated for the same function $f$).

Clearly, every quadrature is associated with a discrete measure concentrated at the nodes of this quadrature.
Therefore, we provide a natural extension of the mentioned theorem, showing that the values of the quadratures of Gauss and Lobatto are extremal in the class of integrals with respect to any measure. The proof is very short and it relies 
on the result from a paper by  Denuit, Lefevre and Shaked
\cite{DLS} concerning stochastic orderings. In our approach the case of even $n$ is not much different, the only difference is that the quadratures  of Gauss and Lobatto are replaced by (left and right) Radau quadratures.

\section{Preliminaries}

In this paper $\mathcal{G}_n[a,b],\mathcal{L}_n[a,b],\mathcal{R}_n^l[a,b]\mathcal{R}_n^r[a,b]$
stand, respectively, for the:
$n-$ point Gauss quadrature operator on the interval $[a,b],$  $n-$point Lobatto quadrature operator on the interval $[a,b],$
$n-$point Radau quadrature operator on the interval $[a,b]$ containing the point $a$ as a node and 
$n-$point Radau quadrature operator on the interval $[a,b]$ containing the point $b$ as a node
(for details see \cite{WG},\cite{WL},\cite{WR}). For example
$$\mathcal{G}_2[a,b](f)=\frac{f\left(\frac{3-\sqrt{3}}{6}a+
	\frac{3+\sqrt{3}}{6}b\right)+f\left(\frac{3+\sqrt{3}}{6}a+
	\frac{3-\sqrt{3}}{6}b\right)}{2},$$
$$\mathcal{L}_3[a,b](f)=\frac16f(a)+\frac23f\left(
\frac{a+b}{2}\right)+\frac16f(b)$$
and 
$$\mathcal{R}^l_2[a,b](f)=\frac14f\left(a\right)+\frac34f\left(\frac{a+2b}{3}\right).$$

Now, we recall the definition of higher order convexity. To this end we need the notion of divided differences of higher orders. These expressions are defined by the following recurrent formulas:  
$$f[x_1]=f(x_1)$$
and
$$f[x_1,\dots,x_n]=\frac{f[x_1,\dots,x_{n-1}]-
	f[x_2,\dots,x_n]}{x_n-x_1}.$$
Let $I\subset\R$ be an interval. We say that function $f$ is convex  of order $n$ if $$f[x_1,\dots,x_{n+2}]\geq 0,\;x_1,\dots,x_{n+2}\in I.$$
Clearly, $0-$convexity means nondecreasingness and $1-$convexity is equivalent to the standard convexity.

In \cite{BP} Bessenyei and P\'ales  proved that for all $n-$convex functions $f$ with odd $n$ we have
\begin{equation}
\label{odd}
\mathcal{G}_{\frac{n+1}{2}}[a,b](f)\leq\frac{1}{b-a}\int_a^bf(x)dx\leq \mathcal{L}_{\frac{n+3}{2}}[a,b](f)
\end{equation}
whereas for even $n$ we have 
\begin{equation}
\label{even}
\mathcal{R}_{\frac{n+2}{2}}^l[a,b](f)\leq\frac{1}{b-a}\int_a^bf(x)dx\leq\mathcal{R}_{\frac{n+2}{2}}^r[a,b](f).
\end{equation}
Note that this result follows from the analysis of the remainder of the quadrature rules involved but such reasoning requires higher regularity of $f$ and, therefore, the result of Bessenyei and P\'ales is stronger than the one stemming from numerical analysis.
On the other hand, Brass and Schmeisser in \cite{BS} proved that 
for odd $n$ we have
\begin{equation}
\label{bsin}
\mathcal{G}_{\frac{n+1}{2}}[a,b](f)\leq\mathcal{Q}[a,b](f)\leq \mathcal{L}_{\frac{n+3}{2}}[a,b](f)
\end{equation}
for every quadrature $Q$ of order $n+1$ with positive weights.

We will generalize  both above results, using the approach connected with use of stochastic orderings. This method is was first used in the paper of   Rajba \cite{Rajba} where, among others, the connection of following Ohlin lemma with Hermite-Hadamard inequality was observed.
\begin{lemma} (Ohlin \cite{Ohlin})
	\label{Ohlin}
Let $X_1,X_2$ be two random variables such that $\E X_1=\E X_2$  and let  $F_1,F_2$ be their cumulative distribution functions.
If $F_1,F_2$ satisfy for some $x_0$ the following inequalities
$$F_1(x)\leq F_2(x)\;\textrm{if}\;x<x_0\;\;\textrm{and}\;\;F_1(x)\geq F_2(x)\;{ if}\;x>x_0 $$
then 
\begin{equation}
\label{m}
\E f(X_1)\leq \E f(X_2)
\end{equation}
 for all continuous and convex functions $f:\R\to\R.$
\end{lemma}
Ohlin lemma and its generalization called Levin-Stechkin theorem  was used recently  in  a number papers (see for example \cite{MM}, \cite{NR}, \cite{NR1}, \cite{OSz}, \cite{Rajba},
\cite{R1}, \cite{R2}, \cite{SzA}, \cite{Szostok2}, \cite{SzostokL},  \cite{SzostokCPAA}).

However, we are interested in inequalities satisfied by higher order convex functions. Therefore, we will use a result proved by  Denuit, Lefevre and  Shaked in \cite{DLS}. First we need the definition of the number 
of crossing points of two given functions.
\begin{definition}
	Let  $F,G:[a,b]\to\R$ be given functions.
	We say that cross exactly $s-$times, at points $x_1<\dots<x_s\in[a,b]$ if 
	$$(-1)^i(F(t)-G(t))\leq 0;\;t\in(x_i,x_{i+1}),i=1,\dots,s$$
	and for each $i$ there exists a point $t_i$
	such that $(-1)^i(F(t_i)-G(t_i))<0$
	or 
	$$(-1)^{i+1}(F(t)-G(t))\leq 0,t\in(x_i,x_{i+1})$$
	and for each $i$ there exists a point $t_i$
	such that $(-1)^{i+1}(F(t_i)-G(t_i))<0.$
\end{definition}
Now, the theorem follows.
\begin{theorem}[\cite{DLS}]
\label{DLS} Let $X$ and $Y$ be two random variables such that
$$
\E(X^j-Y^j)=0,\;\;j=1,2,\dots,s
$$
If the distribution functions $F_X,F_Y$  cross exactly $s-$times, at points $x_1,\dots,x_s$ and 
$$
(-1)^{s+1}(F_Y(t)-F_X(t))\leq 0\;\;for\;all\;t\in[a,x_1]
$$
then
$$\E f(X)\leq\E f(Y)$$
for all $s-$convex functions $f:\R\to\R.$
\end{theorem}
Note that this theorem may be called a higher order counterpart of Ohlin lemma, since in Ohlin lemma 
we had one crossing point and, as a result,  an inequality valid for convex functions.

In our proof we will also need a result from another paper by Bessenyei and P\'ales, \cite{BesPal}.
Before we present it, we need to introduce some notation: if $I\subset\R$
then $\Delta(I)$ and $D(I)$ are defined by the formulas:   $$\Delta(I):=\{(x,y)\in I^2:x\leq y\}$$
and 
$$D(I):=\{(x,x):x\in I\}.$$
Let $\mu$ be a
a non-zero bounded Borel measure on the interval $[0,1]$, then its moments $\mu_n$ are defined by the usual formulas,
$$\mu_n:=\int_0^1t^nd\mu(t),\;n=0,1,2,\dots\;.$$
The main results of \cite{BesPal} is given by the following theorem.
\begin{theorem}[\cite{BesPal}]
	\label{BP} Let $I\subset\R$ be an open interval 	let $\Omega\subset\Delta(I)$ be an open subset
	containing the diagonal $D(I)$ of $I\times I$ and let $\mu$ be a non-zero bounded Borel measure
	on $[0,1]$. Assume that $n$ is the smallest non-negative integer such that $\mu_n\neq 0.$ If
	$f:I\to\R$ is a continuous function satisfying the integral inequality 
	$$\int_0^1f(x+t(y-x))d\mu(t)\geq 0$$
	then $\mu_nf$ is $(n-1)-$convex.
\end{theorem}

\section{Main result}

In this section, following an idea from \cite{SzostokCPAA}, we will prove the
main result of the paper.
\begin{theorem}
\label{mainth}
Let $a,b\in \R, a<b$ and let 
$\mu$ be a probability measure on $\mathcal{B}([a,b])$ such that
\begin{equation}
\label{momenty}
\int_a^bx^kd\mu(x)=\frac{1}{b-a}\int_a^bx^kdx,\;k=1,\dots,n.
\end{equation}
If $n$ is odd then  the inequality
\begin{equation}
\label{muodd}
\mathcal{G}_{\frac{n+1}{2}}[a,b](f)\leq\int_a^bf(x)d\mu(x)\leq \mathcal{L}_{\frac{n+3}{2}}[a,b](f)
\end{equation}
is satisfied for all $n-$convex functions
whereas for even $n$ we have 
\begin{equation}
\label{mueven}
\mathcal{R}_{\frac{n+2}{2}}^l[a,b](f)\leq\int_a^bf(x)d\mu(x)\leq\mathcal{R}_{\frac{n+2}{2}}^r[a,b](f).
\end{equation}
 for all $n-$convex functions $f.$
\end{theorem}
\begin{proof}
As it was mentioned, we will prove this theorem using the method from \cite{SzostokCPAA}.
We start with the first inequality from \eqref{muodd}. Thus we consider the measure 
$$\mu_G=\alpha_1\delta_{x_1}+\dots+\alpha_n\delta_{x_\frac{n+1}{2}}$$
where $\delta_{x_i}$ is a probabilistic measure concentrated at the point $x_i,$
and $\alpha_i,x_i$ are, respectively, weights and nodes of $\mathcal{G}_{\frac{n+1}{2}}[a,b].$
The  quadrature $\mathcal{G}_{\frac{n+1}{2}}[a,b]$
is exact for the monomials of degree not greater than $n,$
i.e.
\begin{equation}
\label{intint}
\int_a^bx^kd\mu_G(x)=\frac{1}{b-a}\int_a^bx^kdx,\;k=1,2,\dots,n.
\end{equation}
Let  $\mu$ satisfy the assumptions of the theorem and let $F$ be the cumulative distribution function connected with  $\mu.$
Further, let $G$ be the cumulative distribution function connected with  $\mu_G$
i.e 
$$G(x)=\left\{\begin{array}{ll}
0&x\in[a,x_1),\\
\alpha_1&x\in[x_1,x_2),\\
\alpha_1+\alpha_2&x\in[x_1,x_2),\\
\vdots\\
1&x>x_{\frac{n+1}{2}}.
\end{array}\right.$$
Observe that the maximal number of crossing points of $F$ and $G$  is equal to $n.$  
Indeed, the first possible crossing point is $x_1.$ Then functions $F$ and $G$ may cross 
at most once in the interval $(x_1,x_2).$ If there were two crossing points in this interval we would have a contradiction with the fact that $F$ is nondecreasing.
Then $F$ and $G$ may cross in the point $x_2$ and so on. Altogether we have $\frac{n+1}{2}$
points and $\frac{n-1}{2}$ intervals and, consequently,
at most
$$\frac{n+1}{2}+\frac{n-1}{2}=n$$
crossing points.

Thus we have an upper estimation of the number of crossing points.  In order to use Theorem \ref{DLS}, it remains to show  that 
the number of crossing points cannot be smaller than $n.$
For the indirect proof, assume that this number equals $k$ for some  $k<n.$ Since, we have $k$ crossing points and $k$ first moments of $\mu$ and $\mu_G$ coincide, we may use Theorem \ref{DLS} and we infer that inequality 
\begin{equation}
\label{contr}
\mathcal{G}_{\frac{n+1}{2}}[a,b](f)\leq\int_a^bf(x)d\mu(x)
\end{equation}
is true for all $k-$convex functions. However, in view of 
\eqref{momenty} and \eqref{intint}, we know that
for some $l\geq n$ 
$$\int_a^bx^kd\mu(x)=\int_a^bx^kd\mu_G(x),\;k=1,\dots,l.$$
and, 
$$\int_a^bx^{l+1}d\mu(x)\neq\int_a^bx^{l+1}d\mu_G(x).$$
Thus from  Theorem \ref{BP}, we 
obtain that every function satisfying \eqref{contr} is $l-$convex.
In consequence,  every $k-$convex function is $l-$convex.
This contradiction shows that the number of crossing points of $F$ and $G$ 
is equal to $n$ and to finish the proof of  it is enough to use Theorem \ref{DLS}.

The proof of 
$$\int_a^bf(x)d\mu(x)\leq \mathcal{L}_{\frac{n+3}{2}}[a,b](f)$$
runs similarly. Let $L$ be the CDF connected with the quadrature $\mathcal{L}_{\frac{n+3}{2}}[a,b].$
The functions $F$ and $L$ can cross in each  node lying in the open interval $(a,b),$ there are ${\frac{n-1}{2}}$
such points.
Further functions $F$ and $L$ can cross at most once  in each interval lying between two consecutive nodes.

In total we have no more than 
$$\frac{n-1}{2}+\frac{n+1}{2}=n$$
crossing points. A similar reasoning as in the previous cases 
shows that there are exactly $n$ crossing points and the 
inequality \eqref{muodd} follows from Theorem \ref{DLS}.

Finally, in the case of even $n$ and of Radau quadratures $$\mathcal{R}_{\frac{n+2}{2}}^l[a,b](f),
 \mathcal{R}_{\frac{n+2}{2}}^r[a,b](f)$$ we have $\frac{n}{2}$
inner nodes and $\frac{n}{2}$ intervals giving the same number 
of possible crossing points. Again, similarly as before, there are exactly $n$ crossing points and the inequality \eqref{mueven}
follows from Theorem \ref{DLS}.
\end{proof}
We end this part of the paper with some simple observations stating that the condition \eqref{momenty} is, in fact, equivalent to the inequalities \eqref{muodd},\eqref{mueven}.
\begin{remark}
\label{remnec}
Theorem \ref{mainth} cannot be further generalized, since condition \eqref{momenty} is necessary if we want any 
of inequalities \eqref{muodd},\eqref{mueven} to be satisfied by all $n-$convex functions.
Indeed, functions of the form $$x\mapsto x,x\mapsto x^2,\dots,x\mapsto x^n$$ are both $n-$convex and $n-$concave. If an inequality of the form 
$$\int_a^bfd\mu\leq\int_a^bfd\nu$$
holds for all $n-$convex functions then for $k=1,2,\dots,n$ we have
$$\int_a^bx^kd\mu(x)\leq\int_a^bx^kd\nu(x)$$
and 
$$\int_a^b-x^kd\mu(x)\leq\int_a^b-x^kd\nu(x)$$
giving us the equality \eqref{momenty}.
\end{remark}
\section{Incomparability}
In a series of papers by W\c{a}sowicz \cite{Wasowicz1,Wasowicz2,Wasowicz3} many inequalities between quadratures were obtained. Most of them (but not all) follow from the result of Brass and Schmeisser (and, in consequence, also from Theorem \ref{mainth}).
For example the inequalities 
 $$\mathcal{G}_2(f)\leq \mathcal{C}_3(f),
\mathcal{L}_4(f)\leq \mathcal{L}_3(f),
\mathcal{L}_5(f)\leq \mathcal{L}_3(f)$$
($\mathcal{C}_3$ means here the three points Chebyshev rule)
for $3-convex$ functions $f,$ yield a very particular cases
of \eqref{muodd},
but the inequalities
$$\mathcal{C}_3(f)\leq\mathcal{G}_3(f),
\mathcal{C}_3(f)\leq\mathcal{L}_4(f),
\mathcal{C}_3(f)\leq\mathcal{L}_5(f),$$
do not follow from Theorem \ref{mainth}.

However, in  W\c{a}sowicz's papers, besides the 
obtained inequalities one can also find negative results, stating that some quadrature operators are not comparable in the class of $n-$convex functions. We observe that in many cases it is enough to check a simple sufficient condition. The following remark is complement to Remark \ref{remnec}.
\begin{remark}
\label{ostrem}
Assume that for some measures $\mu,\nu$  the inequality 
\begin{equation}
\label{abmn}
\int_a^bfd\mu\leq\int_a^bfd\nu
\end{equation}
is satisfied by all $n-$convex functions $f:[a,b]\to\R.$
Then we have 
$$\int_a^bx^jd\mu\leq\int_a^bx^jd\nu,j=1,\dots,n$$
and 
\begin{equation}
\label{n=}
\int_a^bx^{n+1}d\mu\neq\int_a^bx^{n+1}d\nu.
\end{equation}
In view of Remark \ref{remnec} it is enough to show 
the condition \eqref{n=}. 
The proof of \eqref{n=} relies again on Theorem \ref{BP}.
Assume that 
$$\int_a^bx^{n+1}d\mu=\int_a^bx^{n+1}d\nu$$
and let $l$ be the first number for which we have 
$$\int_a^bx^{l+1}d\mu\neq\int_a^bx^{l+1}d\nu.$$
Then on account of Theorem \ref{BP} every function satisfying 
\eqref{abmn} is $l-$convex. This gives a contradiction with the fact that \eqref{abmn} is satisfied by every $n-$convex function.
\end{remark}
Using Remark \ref{ostrem},  it is possible to 
prove all results from \cite{Wasowicz1,Wasowicz2,Wasowicz3} stating 
that some quadratures are incomparable in respective classes of 
functions. For example, in Theorem 2.4 \cite{Wasowicz2}   it is observed that 
$\mathcal{G}_3$ is incomparable in the class of $3-$convex functions
with $\mathcal{L}_4.$ The reason why these operators are incomparable is that 
they have five first moments equal, instead of three. Of course many similar incomparabilities follow from 
Remark \ref{ostrem}. For instance, for all $k,l$ such that $k\geq 3,l\geq 4$ $\mathcal{G}_k$ is incomparable with $\mathcal{L}_l,$ in the class of $3-$convex functions.

The equality of exactly $n$ first moments is of course only a necessary condition for $n-$convex ordering. The sufficient condition is discussed in details in \cite{SzostokCPAA}.

\end{document}